%% file: main.tex
\documentclass{article}
\usepackage{xcolor}
\usepackage{hyperref}
\hypersetup{
  colorlinks   = true, 
  urlcolor     = blue, 
  linkcolor    = green!50!black, 
  citecolor   = red 
}
\usepackage[font=small,labelfont=bf]{caption}
\usepackage{import}
\usepackage{pdfpages}
\usepackage{transparent}

\usepackage{multirow}

\usepackage{authblk}
\usepackage{caption, subcaption}

\usepackage[T2A,T1]{fontenc}    
\usepackage[utf8]{inputenc}     
\usepackage[russian,english]{babel}

\usepackage{mathtools}
\mathtoolsset{showonlyrefs=true}

\DeclareCaptionLabelFormat{subfigure}{\thefigure#2: }
\usepackage{tikz}
\usetikzlibrary {calc}

\usepackage{geometry}

\usepackage{todonotes}
\usepackage{csquotes}

\usepackage{orcidlink}

\usepackage{amsmath,amsfonts,amsthm,amssymb}
\usepackage{bm}

\newtheorem{Definition}{Definition}
\newtheorem{Remark}{Remark}
\newtheorem{Proposition}{Proposition}
\newtheorem{Corollary}{Corollary}

\newtheorem{Lemma}{Lemma}
\newtheorem{Conjecture}{Conjecture}

\providecommand{\keywords}[1]
{
  \small	
  \textbf{\textit{Keywords---}} #1
}

\providecommand{\msc}[1]
{
  \small	
  \textbf{\textit{2020 Mathematics Subject Classification---}} #1
}

\title{On the Eigenvalues of Graphs with Mixed Algebraic Structure}

\author[$\dagger$]{Riccardo Bonetto\orcidlink{0000-0001-8075-6147}}
\author[$\dagger$]{Hildeberto Jardón Kojakhmetov\orcidlink{0000-0001-8708-7409}}

\begin{document}
\affil[$\dagger$]{University of Groningen — Bernoulli Institute for Mathematics, Computer Science and Artificial Intelligence; Groningen, The Netherlands}
\maketitle

\abstract{We study some spectral properties of a matrix that is constructed as a combination of a Laplacian and an adjacency matrix of simple graphs. The matrix considered depends on a positive parameter, as such we consider the implications in different regimes of such a parameter, perturbative and beyond. Our main goal is to relate spectral properties to the graph's configuration, or to basic properties of the Laplacian and adjacency matrices. We explain the connections with dynamic networks and their stability properties, which lead us to state a conjecture for the signature.   }

\vspace{0.5cm}


\vspace{0.5cm}

\msc{15A18, 15A15, 05C50, 05C70, 05C90, 47A55, 34D20}

\input{sections/introduction}
\input{sections/definitions}
\input{sections/dynamics}
\input{sections/conclusions}

\vspace{6pt} 



\appendix
\input{appendix/exp}


\bibliography{bibliography.bib}
\bibliographystyle{abbrv}

\section*{Statements and Declarations}

\begin{description}
    \item[\textbf{Funding}] The authors declare that no funds, grants, or other support were received during the preparation of this manuscript.
    \item[\textbf{Competing Interests}] The authors have no relevant financial or non-financial interests to disclose.
    \item[\textbf{Data Availability}] All datasets are available upon request to the corresponding author.
\end{description}

\end{document}

%% file: sections/introduction.tex
\section{Introduction}

The study of algebraic structures associated with graphs is remarkably relevant for pure and applied mathematics. In the first place, the study of matrices associated with networks can provide information on some features of the associated graph, and vice-versa \cite{godsil2001algebraic, beineke2004topics, biggs1993algebraic}. In addition, in the realm of complex dynamical systems, the linear part of a vector field can be recognised as a matrix related to a network \cite{mirzaev2013laplacian, lambiotte2008laplacian, veerman2020primer, mugnolo2017dynamical, porter2016dynamical, acebron2005kuramoto, strogatz2000kuramoto}. In particular, when studying the stability of the critical points of a vector field, one is especially interested in the eigenvalues of the linearisation at such points. In turn, for matrices associated with graphs, it is often possible, or desirable, to have some connection between the eigenvalues and some graph properties \cite{cvetkovic2009introduction, brouwer2011spectra}. Consequently, we can relate the stability properties of complex dynamical systems to the properties of a graph.

A lot of work has been done to extend the possible algebraic structures associated with a graph, for example, accounting for directed edges or self-loops \cite{hershkowitz1994paths, tadic2001dynamics, suresh2023spectrum}, positive and negative weights \cite{zaslavsky1982signed, bronski2014spectral, akbari2018spectral}, hypergraphs \cite{banerjee2021spectrum, feng1996spectra, cooper2012spectra}, etc. In this manuscript, we study graphs whose set of edges is split into two classes. We assign to one class a Laplacian structure and to the other class an adjacency structure. We propose a matrix, $M(\epsilon)$, that is a combination of such Laplacian and adjacency matrices, where the parameter $\epsilon$ acts as a global weight for the adjacency matrix. Precise definitions will follow in the text. Having in mind the potential application to gradient systems, or Hamiltonian, the Laplacian matrix represents stable interactions, while the adjacency matrix represents unstable saddle-like interactions, see Section \ref{sec:applications}. We exploit some relations between the graphs' properties and their spectrum. In particular, we characterise the perturbative regime to provide conditions for positive definiteness, which in turn can be related to the stability of the aforementioned dynamical systems. Additionally, we provide an upper bound for the number of transitions of the signature of $M(\epsilon)$ as the parameter $\epsilon$ varies. These transitions are important, among other reasons, due to their relationship to bifurcations in dynamical systems. Motivated by the intuition inferred from the dynamics, we state a conjecture for the monotonic change of signature, which we corroborate with a numerical investigation.

The paper is arranged as follows. First, in Section \ref{sec:prel} we introduce the main object of study together with some fundamental properties relevant for the treatment. Section \ref{sec:perturbations} is dedicated to perturbative results, while in Section \ref{sec:bif} we study the signature's transitions. Later, in Section \ref{sec:applications} we exploit the connection with networked dynamical systems. In Appendix \ref{app:prod_fac} we prove a result necessary for a proof (Proposition \ref{prop:transitions}) contained in the main text.

%% file: sections/definitions.tex
\section{Preliminaries}\label{sec:prel}

Let $\mathcal{F}=\{ \mathcal{V}, \mathcal{E}\}$ be a simple graph \cite{west2001introduction}, where $\mathcal{V}=\{ 1, \dots, n \}$ is the set of vertices, and $\mathcal{E}$ is the set of (undirected) edges $\{i,j\} \in \mathcal{E}$.

\begin{Definition}
    The \emph{adjacency matrix}, $A_\mathcal{F}$, of a graph  $\mathcal{F}$ is the $n\times n$ matrix with components
    \begin{equation}
    (A_\mathcal{F})_{ij} :=
        \begin{cases}
            1 \ &\text{if} \ \{i,j\}\in \mathcal{E} , \\
            0 \ &\text{otherwise} .
        \end{cases}
    \end{equation}
\end{Definition}

\begin{Definition}
    The \emph{Laplacian matrix}, $L_\mathcal{F}$, of a graph  $\mathcal{F}$ is the $n\times n$ matrix with components
    \begin{equation}
    (L_\mathcal{F})_{ij} :=
        \begin{cases}
            \deg(i) \ &\text{if} \ i=j , \\
            -1  \ &\text{if} \ \{i,j\}\in \mathcal{E} , \\
            0 \ &\text{otherwise} ,
        \end{cases}
    \end{equation}
    where $ \deg(i)$ is the \emph{degree} of the node $i$, i.e., the number of edges incident to the node $i$.
\end{Definition}

\begin{Definition}\label{def:matrixM}
    Let us consider the subgraphs $\mathcal{G}=\{ \mathcal{V}, \mathcal{E}_\mathcal{G} \}$ and $\mathcal{H}=\{ \mathcal{V}, \mathcal{E}_\mathcal{H} \}$ of $\mathcal{F}$, such that $\mathcal{E}_\mathcal{G} \cap \mathcal{E}_\mathcal{H} =\emptyset $ and $\mathcal{E}_\mathcal{G} \cup \mathcal{E}_\mathcal{H} = \mathcal{E} $. We define the linear operator
\begin{equation}\label{eq:matrix_M}
    M(\epsilon):= L_\mathcal{G} + \epsilon A_\mathcal{H} ,
\end{equation}
where $L_\mathcal{G}$ is the Laplacian matrix of $\mathcal{G}$, $A_\mathcal{H}$ is the adjacency matrix of $\mathcal{H}$, and $\epsilon>0$ is a positive real parameter.
\end{Definition}

Notice that $\epsilon$ acts as global weight on the adjacency matrix. In particular, in the limit $\epsilon=0$ the matrix $M(0)$ becomes the Laplacian matrix $L_\mathcal{G}$, while for $\epsilon \gg 1$ the adjacency matrix $A_\mathcal{H}$ will take the leading role. Such a structure can appear when modelling networked systems with different interactions, in that case the parameter $\epsilon$ can be seen as modulating the \enquote{strength} between the two coupling classes. For the entirety of the paper, otherwise explicitly mentioned, the notation relative to the graphs $\mathcal{F}$, $\mathcal{G}$, $\mathcal{H}$,  and their definitions is given by Definition \ref{def:matrixM}.



\begin{Remark}
    Since $L_\mathcal{G}$ and $A_\mathcal{H}$ are symmetric $n \times n$ matrices, then $M(\epsilon)$ is an $n \times n$ symmetric matrix. Therefore the eigenvalues $\mu_{1}(\epsilon), \dots, \mu_{n}(\epsilon)$ of $M(\epsilon)$ are real, and there exists a complete basis of orthogonal eigenvectors.
\end{Remark}

We can immediately identify a couple of relevant properties of the matrix \eqref{eq:matrix_M}. The first, is that the trace does not depend on the value of the parameter $\epsilon$, but only on the number of edges of $\mathcal{G}$.
\begin{Lemma}
    The trace of $M(\epsilon)$ is given by
    \begin{equation}\label{eq:trace}
        \text{tr}(M(\epsilon))= 2 \# \mathcal{E}_\mathcal{G},
    \end{equation}
    where $\# \mathcal{E}_\mathcal{G}$ denotes the number of edges of $\mathcal{G}$.
\end{Lemma}
\begin{proof}
    The statement follows straightforwardly from the fact that the adjacency matrix is traceless and that the trace of the Laplacian matrix is twice the number of edges of the associated graph. 
\end{proof}

The other property we are going to state is related to the symmetries. We notice that the symmetry properties of the matrix \eqref{eq:matrix_M} are in correspondence with the symmetries of the graph $\mathcal{F}$, where the edges belong now to two classes associated respectively to $\mathcal{G}$ and $\mathcal{H}$.

\begin{Definition}
    A nonsingular matrix $S$ is called a \emph{symmetry} of a matrix $M$ if $S^{-1}MS=M$, or equivalently $[M,S]=0$, where $[\cdot,\cdot]$ denotes the commutator.
\end{Definition}
The set of matrices $S$ that commute with the matrix $M$, together with the matrix product, form the symmetry group of $M$.

\begin{Definition}\label{def:aut}
    Let $\mathcal{F}$ be a simple graph. The \emph{automorphism group} $\textup{aut}(\mathcal{F})$ is the set of permutations leaving the adjacency structure of the graph $\mathcal{F}$ invariant.
\end{Definition}

Let us consider the standard permutation-matrix representation of $\textup{aut}(\mathcal{F})$ in $\mathbb{R}^n$, where $n$ is the number of vertices of $\mathcal{F}$. Since this is the only representation we consider in this paper, we do not explicitly distinguish the group and its representation, so we write $S \in \textup{aut}(\mathcal{F})$. As a consequence of Definition \ref{def:aut} we have that $[S, A_\mathcal{F}]=0$, $\forall S \in \textup{aut}(\mathcal{F})$. Similarly, the commutation relations hold also for the Laplacian matrix.

\begin{Proposition}
    The group $\textup{aut}(\mathcal{G}) \cap \textup{aut}(\mathcal{H})$ is the symmetry group of the matrix $M(\epsilon)$.
\end{Proposition}

\begin{proof}
 Let $S \in \textup{aut}(\mathcal{G}) \cap \textup{aut}(\mathcal{H})$ then, by definition, we have $[S, L_\mathcal{G}]=0$, $[S, A_\mathcal{H}]=0$. Therefore $[S,  L_\mathcal{G} + \epsilon A_\mathcal{H}] = [S,  L_\mathcal{G} ]+ \epsilon [S, A_\mathcal{H}]  =0$, $\forall S \in \textup{aut}(\mathcal{G}) \cap \textup{aut}(\mathcal{H})$, and $\forall \epsilon$.
\end{proof}

\begin{Proposition}
    The matrix $L_\mathcal{G}$ is semi-positive definite, with $\dim(\ker(L_\mathcal{G}))$ equal to the number of connected components of $\mathcal{G}$. Let $\mathcal{G}_1, \dots, \mathcal{G}_r$ be the connected components of $\mathcal{G}$, then the eigenvectors of the zero eigenvalues are 
    \begin{equation}
        \mathbf{1}_{\mathcal{G}_1}:=(\underbrace{1,\cdots,1}_{\# \mathcal{V}_{\mathcal{G}_1}}, 0, \cdots, 0)^\intercal, \mathbf{1}_{\mathcal{G}_2}:=(0, \cdots,0,\underbrace{1,\cdots,1}_{\# \mathcal{V}_{\mathcal{G}_2}}, 0, \cdots, 0)^\intercal, \dots, \mathbf{1}_{\mathcal{G}_r}:=(0, \cdots,0,\underbrace{1,\cdots,1}_{\# \mathcal{V}_{\mathcal{G}_r}})^\intercal ,
    \end{equation}
    where $\# \mathcal{V}_{\mathcal{G}_i}$ is the number of nodes in the connected component $\mathcal{G}_i$, $i=1, \dots, r$.
\end{Proposition}

\begin{Definition}
    Let $\mathcal{G}$ be a graph with $r$ connected components. The \emph{normalised eigenvectors} associated with the zero eigenvalues of $L_\mathcal{G}$ are given by
    \begin{equation}\label{eq:normalised_eig}
        \hat{ \mathbf{1}}_{\mathcal{G}_i}:= \frac{\mathbf{1}_{\mathcal{G}_i}}{\sqrt{\# \mathcal{V}_{\mathcal{G}_i}}} ,
    \end{equation}
    where $i =1, \dots, r$.
\end{Definition}

\begin{Definition}\label{def:inertia}
    The \emph{inertia} of a matrix $Q$ is the ordered triplet $\{ N_-(Q) , N_0(Q) , N_+(Q) \}$, where $N_-(Q) , N_0(Q) , N_+(Q)$ are respectively the number of negative, zero and positive eigenvalues of $Q$ counted with multiplicity.
\end{Definition}

\begin{Definition}\label{def:sign}
    The \emph{signature} of a matrix $Q$ is the number $s(Q)=N_+(Q) - N_-(Q)$.
\end{Definition}

\begin{Remark}
    For nonsingular symmetric matrices, the signature uniquely identifies the inertia and vice-versa. Moreover, thanks to Sylvester's law of inertia \cite{sylvester1852xix, parlett1998symmetric}, the signature and the inertia of a matrix are basis independent.
\end{Remark}


\subsection{Perturbation theory}\label{sec:perturbations}

In this section, we study the spectral properties of $M(\epsilon)$ in the regime of small perturbations, $0<\epsilon\ll1$. Let us expand the eigenvalues of $M(\epsilon)$ in asymptotic series, 
\begin{equation}
    \mu_i(\epsilon) =  \mu_i^{(0)} + \epsilon  \mu_i^{(1)} + \epsilon^2  \mu_i^{(2)} + \textup{h.o.t.}  ,
\end{equation}
where $i =1, \dots, n$. Notice that the zeroth-order terms are the eigenvalues of the Laplacian matrix $L_\mathcal{G}$, i.e., $\mu_i^{(0)}=\lambda_i$, where $\lambda_i \in \text{spec}(L_\mathcal{G})$. 

\begin{Proposition}\label{prop:positive_stay_positive}
    Let $\mathcal{G}$ be a graph with $n$ vertices and $r$ connected components. Then there are at least $n-r$ eigenvalues of $M(\epsilon)$ that, for $\epsilon$ sufficiently small, are positive reals. 
\end{Proposition}

\begin{proof}
    Since $\mathcal{G}$ has $r$ connected components, then $L_\mathcal{G}$ has $r$ repeated zero eigenvalues and $n-r$ positive eigenvalues. Thanks to Gershgorin Theorem \cite{gershgorin1931uber, stewart1990matrix}, for $\epsilon$ small enough the positive eigenvalues will remain positive.
\end{proof}

Let us notice that Proposition \ref{prop:positive_stay_positive} holds for any small symmetric perturbation of $L_\mathcal{G}$. The situation is more delicate when we study the first-order perturbation of the zero eigenvalues of $L_\mathcal{G}$. In such case, it is relevant to consider the actual structure of the perturbation $\epsilon A_\mathcal{H}$.

\begin{Corollary}\label{cor:generic_pert}
    Let $\mathcal{G}$ be a graph with $r$ connected components, $\mathcal{G}_1, \dots, \mathcal{G}_r$. Let $\mu_1(\epsilon), \dots, \mu_r(\epsilon)$ be the eigenvalues of $M(\epsilon)$ with zeroth-order term equal to zero, i.e., $\mu_1^{(0)}= \dots =\mu_r^{(0)}=0$. The first-order corrections are $\mu_1^{(1)}= \epsilon \theta_1, \dots, \mu_r^{(1)}=\epsilon \theta_r$, where $\theta_1, \dots, \theta_r$ are the eigenvalues of the matrix
    \begin{equation}\label{eq:matrix_first_order_perturbations}
        \begin{pmatrix}
             \dfrac{2\# \mathcal{E}_\mathcal{H}^{(\mathcal{G}_1)}}{\#\mathcal{V}_{\mathcal{G}_1}} & \dfrac{\# \mathcal{E}_\mathcal{H}^{(\mathcal{G}_1,\mathcal{G}_2)}}{\sqrt{\#\mathcal{V}_{\mathcal{G}_1}\#\mathcal{V}_{\mathcal{G}_2}}} & \dots & \dfrac{\# \mathcal{E}_\mathcal{H}^{(\mathcal{G}_1,\mathcal{G}_r)}}{\sqrt{\#\mathcal{V}_{\mathcal{G}_1}\#\mathcal{V}_{\mathcal{G}_r}}} \\[4ex]
            \dfrac{\# \mathcal{E}_\mathcal{H}^{(\mathcal{G}_2,\mathcal{G}_1)}}{\sqrt{\#\mathcal{V}_{\mathcal{G}_2}\#\mathcal{V}_{\mathcal{G}_1}}} & \dfrac{2\# \mathcal{E}_\mathcal{H}^{(\mathcal{G}_2)}}{\#\mathcal{V}_{\mathcal{G}_2}} & \dots & \dfrac{\# \mathcal{E}_\mathcal{H}^{(\mathcal{G}_2,\mathcal{G}_r)}}{\sqrt{\#\mathcal{V}_{\mathcal{G}_2}\#\mathcal{V}_{\mathcal{G}_r}}}  \\[4ex]
            \vdots & \vdots & \ddots & \vdots \\[4ex]
            \dfrac{\#\mathcal{E}_\mathcal{H}^{(\mathcal{G}_r,\mathcal{G}_1)}}{\sqrt{\#\mathcal{V}_{\mathcal{G}_r}\#\mathcal{V}_{\mathcal{G}_1}}} & \dfrac{\# \mathcal{E}_\mathcal{H}^{(\mathcal{G}_r,\mathcal{G}_2)}}{\sqrt{\#\mathcal{V}_{\mathcal{G}_r}\#\mathcal{V}_{\mathcal{G}_2}}} & \dots &  \dfrac{2\# \mathcal{E}_\mathcal{H}^{(\mathcal{G}_r)}}{\#\mathcal{V}_{\mathcal{G}_r}} 
        \end{pmatrix} ,
    \end{equation}
    where $\# \mathcal{E}_\mathcal{H}^{(\mathcal{G}_i)}$ is the number of edges of $\mathcal{H}$ connecting nodes of $\mathcal{G}_i$, and $\# \mathcal{E}_\mathcal{H}^{(\mathcal{G}_i,\mathcal{G}_j)}$ is the number of edges of $\mathcal{H}$ connecting a node of $\mathcal{G}_i$ to a node of $\mathcal{G}_j$.
\end{Corollary}
\begin{proof}
    We need to find the first-order correction for the repeated zero eigenvalues of $L_\mathcal{G}$. Since the matrix $L_\mathcal{G}$ is symmetric we have $r$ distinct orthogonal eigenvectors associated with the zero eigenvalues \eqref{eq:normalised_eig}. So, by applying the results of \cite{seyranian1993sensitivity} for weak interactions, we find that the equation to be solved is 
    \begin{equation}
        \det\left( \langle A_\mathcal{H} \hat{ \mathbf{1}}_{\mathcal{G}_i}, \hat{ \mathbf{1}}_{\mathcal{G}_j} \rangle - \theta \delta_{ij} \right) ,
    \end{equation}
    where $i,j = 1, \dots, r$, $\delta_{ij}$ is the Kronecker delta, and $\langle \cdot , \cdot  \rangle$ is the standard scalar product in $\mathbb{R}^n$. By computing explicitly the matrix elements $\langle A_\mathcal{H} \hat{ \mathbf{1}}_{\mathcal{G}_i}, \hat{ \mathbf{1}}_{\mathcal{G}_j} \rangle$ we retrieve the matrix \eqref{eq:matrix_first_order_perturbations}.
\end{proof}

\begin{Lemma}\label{lm:two_cc}
    Let $\mathcal{G}$ be a graph with two connected components, $\mathcal{G}_1$ and $\mathcal{G}_2$. For $\epsilon$ sufficently small, the matrix $M(\epsilon)$ is positive definite iff
    \begin{equation}
        \left(\# \mathcal{E}_\mathcal{H}^{(\mathcal{G}_1,\mathcal{G}_2)}\right)^2< 4 \# \mathcal{E}_\mathcal{H}^{(\mathcal{G}_1)} \# \mathcal{E}_\mathcal{H}^{(\mathcal{G}_2)}.
    \end{equation}
\end{Lemma}

\begin{proof}
    For a graph $\mathcal{G}$ with two connected components the matrix \eqref{eq:matrix_first_order_perturbations} reduces to
    \begin{equation}
        \begin{pmatrix}\label{eq:two_dim_matrix}
             \frac{2\# \mathcal{E}_\mathcal{H}^{(\mathcal{G}_1)}}{\#\mathcal{V}_{\mathcal{G}_1}} & \frac{\mathcal{E}_\mathcal{H}^{(\mathcal{G}_1,\mathcal{G}_2)}}{\sqrt{\#\mathcal{V}_{\mathcal{G}_1}\#\mathcal{V}_{\mathcal{G}_2}}} \\
            \frac{\mathcal{E}_\mathcal{H}^{(\mathcal{G}_2,\mathcal{G}_1)}}{\sqrt{\#\mathcal{V}_{\mathcal{G}_2}\#\mathcal{V}_{\mathcal{G}_1}}} & \frac{2\# \mathcal{E}_\mathcal{H}^{(\mathcal{G}_2)}}{\#\mathcal{V}_{\mathcal{G}_2}}
        \end{pmatrix} =: 
        \begin{pmatrix}
            a & c \\
            c & b
        \end{pmatrix}.
    \end{equation}
    The eigenvalues of \eqref{eq:two_dim_matrix} are 
    \begin{equation}
        \theta_{1,2} = \frac{a+b \pm \sqrt{(a-b)^2 + 4 c^2}}{2} .
    \end{equation}
    Clearly, the first eigenvalue, $\theta_1= (a+b + \sqrt{(a-b)^2 + 4 c^2})/2$, is always positive as all the terms are positive. The second eigenvalue can be rewritten as $\theta_2= (a+b - \sqrt{(a+b)^2 - 4 (ab- c^2)})/2$, which is positive only if $c^2<ab$. The statement follows.
\end{proof}

\begin{Corollary}\label{cor:one_cc}
    Let $\mathcal{G}$ be a connected graph and $\mathcal{E}_\mathcal{H}$ non-empty, then, for $\epsilon$ sufficiently small, the matrix $M(\epsilon)$ is positive definite. Moreover, the first-order expansion of the zero eigenvalue of $L_\mathcal{G}$ is
    \begin{equation}
        \mu_0(\epsilon) = \epsilon \frac{2 \#\mathcal{E}_\mathcal{H}}{n} + O(\epsilon^2) .
    \end{equation}
\end{Corollary}

\begin{proof}
    The statement follows from the explicit computation of a Rayleigh quotient \cite{parlett1998symmetric}, or simply from the one-dimensional component case of Corollary \ref{cor:generic_pert}.
\end{proof}


\subsection{Bifurcation of eigenvalues}\label{sec:bif}

When the parameter $\epsilon$ is not restricted to small values, the study of the eigenvalues of $M(\epsilon)$ becomes more complicated. Indeed, by varying $\epsilon$ in the domain of positive reals, one can ask how many times the eigenvalues change sign. Such information is encoded in the inertia and in the signature of the matrix $M(\epsilon)$, see Definitions \ref{def:inertia}, \ref{def:sign}. In the following proposition, we provide an upperbound on the changes of signs for the eigenvalues of $M(\epsilon)$.

\begin{Proposition}\label{prop:transitions}
    Let $\epsilon>0$, and $s(M(\epsilon)) = N_+(M(\epsilon)) - N_-(M(\epsilon))$ be the signature of the matrix $M(\epsilon)$. The maximum number of transitions of signature is
    \begin{equation}\label{eq:upperbound}
        n-\dim(\ker(A_{\mathcal{H}})) .
    \end{equation}
\end{Proposition}

\begin{proof}
    The signature can change only if the matrix $M(\epsilon)$ is singular for some value of $\epsilon$. In turn, $M(\epsilon)$ is singular if $\det(M(\epsilon))=0$. In order to find an upper bound in the number of transitions of the signature, we need to study the maximal power in $\epsilon$ of the polynomial $\det(M(\epsilon))$. Let $M_{ij}(\epsilon)$ be the components of the matrix $M(\epsilon)$, then the determinant can be written as
    \begin{equation}
        \det(M(\epsilon)) = \sum_{\sigma \in \mathbb{S}_n} \text{sgn}(\sigma) \prod_i^n \left(M(\epsilon)\right)_{i\sigma(i)} ,
    \end{equation}
    where $\mathbb{S}_n$ is the symmetric group of permutations over $n$ elements, and $\text{sgn}(\sigma)$ is the sign of the permutation $\sigma$. We recall that $M(\epsilon)$ is, by definition, $M(\epsilon)=L_\mathcal{G} + \epsilon A_\mathcal{H}$. So, we have
    \begin{align}
        \det(M(\epsilon)) = \sum_{\sigma \in \mathbb{S}_n} \text{sgn}(\sigma) \prod_i^n \left( \left({L_\mathcal{G}}\right)_{i\sigma(i)} + \epsilon \left({A_\mathcal{H}}\right)_{i\sigma(i)} \right)  . 
    \end{align}
    We expand the product following the factorisation described in Appendix \ref{app:prod_fac}, obtaining 
    \begin{align} 
    &\det({L_\mathcal{G}}) +    \\
        & + \sum_{\sigma \in \mathbb{S}_n} \text{sgn}(\sigma) \bigg( \sum_{r=1}^{n-1} \epsilon^r C_{n,r}\left( \left({L_\mathcal{G}}\right)_{1\sigma(1)}, \dots,  \left({L_\mathcal{G}}\right)_{n\sigma(n)};  \left({A_\mathcal{H}}\right)_{1\sigma(1)}, \dots, \left({A_\mathcal{H}}\right)_{n\sigma(n)} \right) \bigg)  + \\
        &+ \epsilon^n  \det(A_\mathcal{H}) ,
    \end{align}
    where $C_{n,r}$ is defined in \eqref{eq:c_def}, and roughly speaking it is the sum of the non-repeated combination of products where the terms $\left({L_\mathcal{G}}\right)_{i\sigma(i)}$ appear $n-r$ times and  $\left({A_\mathcal{H}}\right)_{j\sigma(j)}$ appear $r$ times, $i\neq j$.
    We recall that the Laplacian matrix of a simple graph is singular, so $\det({L_\mathcal{G}})=0$, and therefore one zero is always at $\epsilon=0$. This fact is easy to check also without the above expansion. In order to have a non-zero term with power $r$ in $\epsilon$ it is necessary that at least one term of the form
    \begin{equation}
        \left({A_\mathcal{H}}\right)_{k_1\sigma(k_1)} \cdots \left({A_\mathcal{H}}\right)_{k_r\sigma(k_r)} ,
    \end{equation}
    where $k_1, \dots ,k_r$ are non-repeated combinations of elements in $\{1, \dots, n \}$, is non-zero. 
    %
    %
    For that to be true, the matrix $A_\mathcal{H}$ needs to have $r$ linearly independent eigenvectors. So, the maximum power of $\epsilon$ is bounded by the number of linearly independent eigenvectors of $A_\mathcal{H}$, which in turn is equal to the number of non-zero eigenvalues of $A_\mathcal{H}$, i.e., $n- \dim(\ker(A_\mathcal{H}))$. 
\end{proof}

%% file: sections/dynamics.tex
\section{Qualitative Study of Linear Vector Fields}\label{sec:applications}

The treatment carried out in this paper is also related to the study of networked dynamical systems. Let us assign to each edge of the graph $\mathcal{F}$ a potential function depending on the position coordinates associated to the connected vertices, i.e., $\{i,j\} \mapsto V(q_i, q_j)$. The potential represents the interaction between the state variables defined on the vertex set. The total potential is given by
\begin{equation}
    V_\mathcal{F} := \sum_{\{i,j\} \in \mathcal{E}} V(q_i, q_j) .
\end{equation}
We consider, for example, gradient vector fields
\begin{equation}\label{eq:grad}
    \dot{q} = - \nabla_q V_\mathcal{F} ,
\end{equation}
where $\nabla_{q}:=(\partial/\partial_{q_1}, \dots, \partial/\partial_{q_n})^\intercal$ is a differential operator. 
%
The vector field \eqref{eq:grad} inherits some properties and relations to the graph structure on which it is defined. Following the construction outlined at the beginning of the paper, Section \ref{sec:prel}, we split the interaction into two parts
\begin{equation}
V(q_i, q_j) = 
    \begin{cases}
        \frac{1}{2} (q_i - q_j)^2 \ &\text{if} \ \{i,j\} \in \mathcal{E}_\mathcal{G} , \\
       \epsilon q_i q_j \ &\text{if} \ \{i,j\} \in \mathcal{E}_\mathcal{H} .
    \end{cases}
\end{equation}
Now, if we explicitly compute equation \eqref{eq:grad} we get
\begin{align}\label{eq:eq_exp}
    \dot{q} &= - M(\epsilon) q,
\end{align}
where the matrix $M(\epsilon)$ is the one defined in \eqref{eq:matrix_M}, which reads $M(\epsilon)=L_\mathcal{G} + \epsilon A_\mathcal{H}$. Let us call the interactions of the form $ (q_i - q_j)^2/2$ \emph{diffusive}, while the interactions of the form $q_i q_j$ we call \emph{saddle}. Such terminology comes from the dynamical behaviour of the interactions when analysed separately. On the one hand, diffusive interactions give rise to the Laplacian matrix, that represent a linear diffusion process on the graph. For a connected graph, the system subject to diffusion converges to the mean value of the initial conditions. So, the interaction is stable and attracting, with a line of of equilibria representing the \enquote{thermalisation} values. Note that the line of equilibria is given by the eigenspace associated to the zero eigenvalue of the Laplacian. On the other hand, the term saddle is used to classify the behaviour of systems that have both positive and negative eigenvalues. Clearly, the adjacency matrix being traceless has positive and negative eigenvalues. The behaviour close to a saddle is generically unstable, as the only stable lines are given by the eigenspaces associated with the positive eigenvalues of the adjacency matrix.

An interesting and surprising consequence of Corollary \ref{cor:generic_pert}, Lemma \ref{lm:two_cc}, and Corollary \ref{cor:one_cc} is that, under some conditions related to the graph structure, the addition of \emph{weak}, i.e., $\epsilon\ll 1$, saddle interactions can further stabilise the system. As a particular application of Lemma \ref{lm:two_cc}, we have the following:

\begin{Proposition}\label{prop:2dyn}
 Let $\mathcal F$ be a simple graph and $\{\mathcal G,\mathcal H\}$ subgraphs of $\mathcal F$ as defined in Section \ref{sec:prel}. Let $\mathcal G$ have two connected components, namely $\mathcal{G}_1=\{\mathcal V_{\mathcal{G}_1},\mathcal E_{\mathcal{G}_1}\}$ and $\mathcal{G}_2=\{\mathcal V_{\mathcal{G}_2},\mathcal E_{\mathcal{G}_2}\}$ and $\mathcal E_{\mathcal H}$ be non-empty. Then:
 \begin{itemize}
     \item If $\mathcal{G}_1$ or $\mathcal{G}_2$ is a complete graph, then, for $\epsilon>0$ sufficiently small, the gradient dynamics $\dot{q} = - \nabla_q V_\mathcal{F} $, as defined above, is unstable. 
     \item Let both $\mathcal{G}_i$ be non-complete graphs. Let $s_i$ be the number of saddle interactions \emph{within} vertices in $\mathcal{G}_i$, $i=1,2$, and $s$ be the number of saddle interactions \emph{between} vertices in $\mathcal{G}_1$ and $\mathcal{G}_2$ (we recall that these saddle interactions are defined through the edge set of $\mathcal H$). If $s^2<s_1s_2$, then, for $\epsilon>0$ sufficiently small, the gradient dynamics $\dot{q} = - \nabla_q V_\mathcal{F} $, as defined above, is stable. 
 \end{itemize}
\end{Proposition}


\begin{Remark}
    We emphasise that a similar result as in Proposition \ref{prop:2dyn} can be stated for more than two connected components, see Corollary \ref{cor:generic_pert}.
\end{Remark}


Increasing $\epsilon$ leads to possible transitions of stability, see Proposition \ref{prop:transitions}. Moreover, when $\epsilon$ increases the behaviour induced by the adjacency matrix will become more and more dominant. In turn, this would imply an increasing predominance of the unstable interactions associated to the negative eigenvalues of $A_\mathcal{H}$. Essentially, the dimension of the unstable eigenspace grows as $\epsilon$ increases. For such reasons we state the following conjecture.

\begin{Conjecture}\label{conj:monotone}
    The signature of the matrix $M(\epsilon)$ is a monotonic decreasing function of $\epsilon$.
\end{Conjecture}

We corroborate the statement of Conjecture \ref{conj:monotone} by a numerical investigation, Figure \ref{fig:tab_random_graphs}. Moreover, we can observe that the conjecture can be extended to any matrix of the form $A+ \epsilon B$, where $A$ is a symmetric real positive semidefinite matrix, and $B$ is a symmetric real matrix, Figure \ref{fig:tab_random_matrices}. Let us notice that if we assume the conjecture to be correct, we can find another bound for the number of transitions of the signature of $M(\epsilon)$, complementing the result of Proposition \ref{prop:transitions}. In fact, for the limit $\epsilon \to \infty$ the signature of $M(\epsilon)$ becomes the signature of the adjacency matrix $A_\mathcal{H}$, which we denote by $s(A_\mathcal{H})$. In the regime of $\epsilon \gg 1$, we can consider the perturbation problem  $A_\mathcal{H} + (1/\epsilon) L_\mathcal{G}$. Let us recall that $A_\mathcal{H}$ can have some zero eigenvalues that under the perturbation of $L_\mathcal{G}$ will generically perturb to some nonzero real number. Since we look for an upper bound in the number of transitions, we should consider the case where all the zero eigenvalues become negative under perturbation. So, the upper bound is
\begin{equation}\label{eq:upp_after}
    n- s(A_\mathcal{H})+\dim(\ker(A_\mathcal{H})) .
\end{equation}
Then, combining \eqref{eq:upp_after} and \eqref{eq:upperbound} we have that an upper bound for the number of transitions of signature is given by
\begin{equation}
    \min\left\{ n-\dim(\ker(A_{\mathcal{H}})) ,  n- s(A_\mathcal{H})+\dim(\ker(A_\mathcal{H})) \right\} .
\end{equation}
The best bound is obtained by examining these two conditions
\begin{align}
    &\text{if} \  2 \dim(\ker(A_{\mathcal{H}})) > s(A_\mathcal{H}) \ \implies \ n-\dim(\ker(A_{\mathcal{H}})) ,\\
    &\text{if} \  2 \dim(\ker(A_{\mathcal{H}})) < s(A_\mathcal{H}) \ \implies \ n- s(A_\mathcal{H})+\dim(\ker(A_\mathcal{H})) ,
\end{align}
and, of course, if the equality $ 2 \dim(\ker(A_{\mathcal{H}})) = s(A_\mathcal{H})$ holds then they are both equal. In Figure \ref{fig:bounds_transitions} we show a comparison between the number of transitions detected and the upper bounds we derived.

\begin{figure}[ht]
\centering
\includegraphics[width=\textwidth]{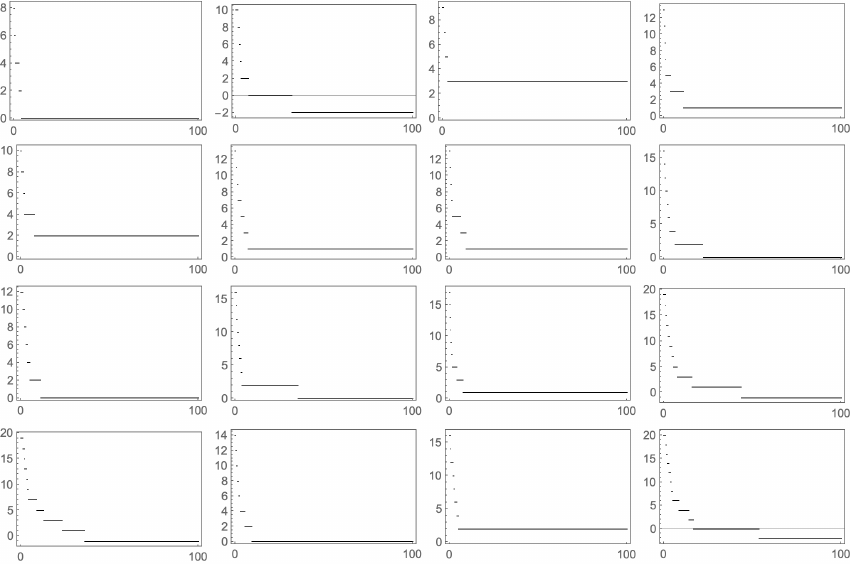}
\caption{ A sample of $16$ plots displaying the signature of the matrix $M(\epsilon)$ on the vertical axis with respect to the parameter $\epsilon$ on the horizontal axis. 
The plots are obtained by generating a random graph with a random number of nodes and edges in the ranges $\# \mathcal{V} \in [10,20]$ and $\# \mathcal{E} \in [2 \# \mathcal{V},\# \mathcal{V}(\# \mathcal{V}-1)/2]$. All distributions are uniform. The parameter $\epsilon$ goes from $1/100$ to $100$, with a sampling rate of $1/100$.
}
\label{fig:tab_random_graphs}
\end{figure}

\begin{figure}[ht]
\centering
\includegraphics[width=\textwidth]{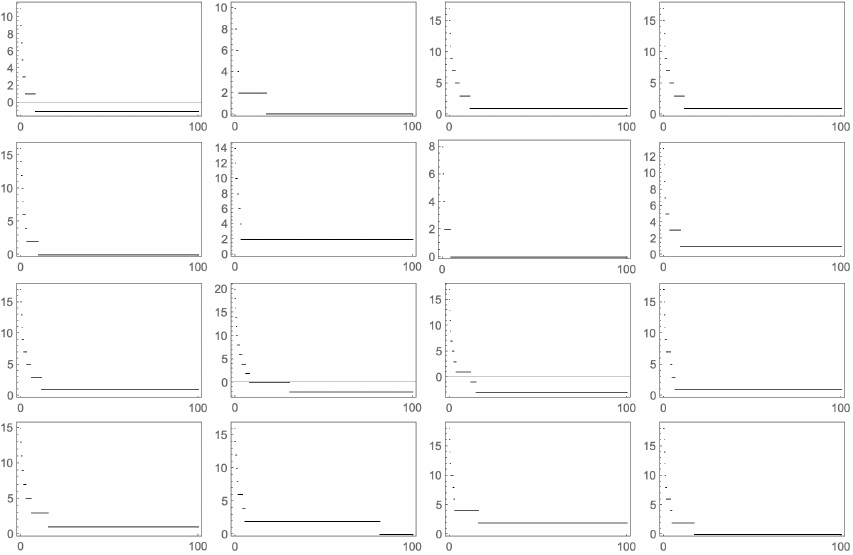}
\caption{ A sample of $16$ plots, signature vs $\epsilon$, for matrices of the from $A+ \epsilon B$, where $A$ is a symmetric real positive semidefinite matrix, and $B$ is a symmetric real matrix. We first generate two square matrices $a$ and $b$, where the entries are random reals uniformly distributed in the interval $[-1,1]$. The dimension of the matrices is given by a random integer between $10$ and $20$, the distribution is once again uniform. In order to obtain the aforementioned properties for the matrices we set $A= a a^\intercal$ and $B=b +b^\intercal$. The parameter $\epsilon$ goes from $1/100$ to $100$, with a sampling rate of $1/100$. }
\label{fig:tab_random_matrices}
\end{figure}

\begin{figure}[ht]
\centering
\includegraphics[width=.7\textwidth]{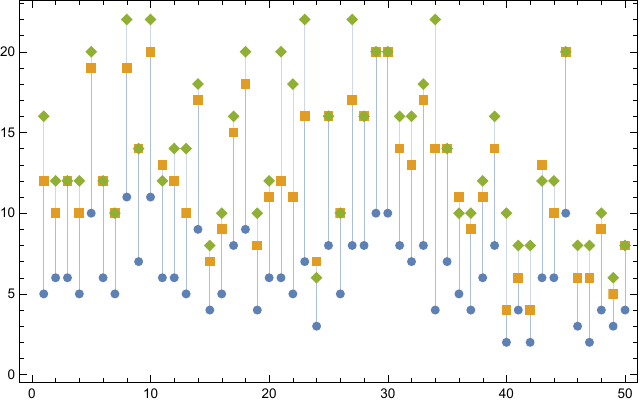}
\caption{Numerical visualisation and comparison between the number of transitions detected and the upper bounds we derived. We employed $50$ randomly generated graphs with a random number of nodes $\# \mathcal{V}$ in the interval $[5, 20]$, and a random number edges $\# \mathcal{E}$ in the interval $[2 \# \mathcal{V}, \mathcal{V} (\mathcal{V}-1)/2]$. All distributions are uniform. For each edge of the graph, we assigned with probability $1/2$ a class that would correspond respectively to the Laplacian or the adjacency structure. The values of $\epsilon$ considered range in the interval $[1/100,100]$ with a sampling rate of $1/100$. The vertical axis shows the number of transitions, while the horizontal one covers the test performed. The blue circles represent the effective number of transitions detected; the orange squares and the green rhombuses represent respectively the bound $n-\dim(\ker(A_{\mathcal{H}}))$ and $n- s(A_\mathcal{H})+\dim(\ker(A_\mathcal{H}))$. 
}
\label{fig:bounds_transitions}
\end{figure}

%% file: sections/conclusions.tex
\section{Outlook}

In this paper, we introduced the matrix \eqref{eq:matrix_M}, that we called $M(\epsilon)$, representing a mixing between the Laplacian and Adjacency matrix. We studied how some spectral properties are related to the graph's configurations. In turn, we exploited the connection with networks of dynamical systems, in particular with gradient systems. Similarly, one can also derive a connection with Hamiltonian systems, where the stability properties are affected by the symplectic structure. Of course, there are possible generalisations of the definitions we stated. Potential directions include considering other graph structures, e.g., directed graphs, or hypergraphs, but also other combinations of known algebraic structures. It might also be possible to explore in more detail other properties of the matrix $M(\epsilon)$, such as its eigenvectors. Moreover, it would be quite significant to prove or further understand Conjecture \ref{conj:monotone}.

%% file: appendix/exp.tex
\section{Factorisation of the product $\prod_i (x_i + \epsilon y_i)$ }\label{app:prod_fac}

We reserve this appendix for the factorisation in powers of $\epsilon$ of the product
\begin{equation}\label{eq:prod}
    \prod_{i=1}^n (x_i + \epsilon y_i).
\end{equation}
The factorisation turns out suitable for the computation of the determinant of $M(\epsilon)$, see Section \ref{sec:bif}. In general, the expansion proposed provide the standard polynomial-in-$\epsilon$ form for the product \eqref{eq:prod}.

\begin{Proposition}
    The product \eqref{eq:prod} can be rewritten as a polynomial of degree $n$ in $\epsilon$ of the form
    \begin{equation}
        \prod_{i=1}^n (x_i + \epsilon y_i) = \sum_{r=0}^n \epsilon^r C_{n,r}(x_1, \dots, x_n; y_1, \dots, y_n) ,
    \end{equation}
    where $C_{n,r}(x_1, \dots, x_n; y_1, \dots, y_n)$ is defined as
    \begin{equation}\label{eq:c_def}
        C_{n,r}(x_1, \dots, x_n; y_1, \dots, y_n):= \underbrace{ x_1 \cdots x_{n-r} y_{n-r+1} \cdots y_n + \dots + y_1 \cdots y_r x_{r+1} \cdots x_{n}}_{\binom{n}{r} \ \text{possible combinations}} .
    \end{equation}
\end{Proposition}

\begin{proof}
    We start by expanding \eqref{eq:prod} as follows
    \begin{align}
        \prod_{i=1}^n (x_i + \epsilon y_i) &= \sum_{k_1, \dots, k_n =0}^1 \prod_{i=1}^n x^{k_i} (\epsilon y)^{1-k_i} \\
        &= \sum_{k_1, \dots, k_n =0}^1 \epsilon^{n- \sum_{i=1}^n k_i} \prod_{i=1}^n x^{k_i} y^{1-k_i} .
    \end{align}
    Notice that the exponents $k_i$ and $1-k_i$ are combinations of $0$ and $1$. So the sum $\sum_{k_1, \dots, k_n =0}^1$ gives the $2^n$ combinations of products $\prod_{i=1}^n x^{k_i} y^{1-k_i}$. We group together the combinations that have the same power of $\epsilon$, i.e., $n- \sum_{i=1}^n k_i =r$, $r=0,\dots, n$. For a given $r$ we have that there are $\binom{n}{r}$ products that are non-repeated combinations of $n-r$ elements of $\{ x_1, \dots, x_n\}$ and $r$ elements of  $\{ y_1, \dots, y_n\}$.
\end{proof}